\documentclass[a4,12pt]{amsart}
\usepackage{amssymb,amstext,amsmath,amscd,amsthm,amsfonts,enumerate,graphicx,latexsym,color}
\newtheorem{thm}{Theorem}[section]
\newtheorem{cor}[thm]{Corollary}
\newtheorem{lem}[thm]{Lemma}

\newtheorem*{thm*}{Theorem}
\theoremstyle{definition}
\newtheorem{dfn}[thm]{Definition}
\newtheorem{rem}[thm]{Remark}

\newtheorem{conv}[thm]{Convention}

\newtheorem*{claim*}{Claim}

\theoremstyle{remark}
\newtheorem*{ac}{Acknowledgments}

\def\p{\mathfrak{p}}

\def\m{\mathfrak{m}}

\def\xx{\text{{\boldmath$x$}}}

\def\spec{\operatorname{\mathsf{Spec}}}
\def\supp{\operatorname{\mathsf{Supp}}}
\def\ssupp{\operatorname{\mathsf{Supp}_{\mathrm{sg}}}}
\def\sing{\operatorname{\mathsf{Sing}}}
\def\res{\operatorname{\mathsf{res}}}
\def\thick{\operatorname{\mathsf{thick}}}
\def\mod{\operatorname{\mathsf{mod}}}
\def\ipd{\operatorname{\mathsf{IPD}}}
\def\dim{\operatorname{\mathsf{dim}}}
\def\proj{\operatorname{\mathsf{proj}}}
\def\add{\operatorname{\mathsf{add}}}
\def\Rfd{\operatorname{\mathsf{Rfd}}}
\def\depth{\operatorname{\mathsf{depth}}}
\def\pd{\operatorname{\mathsf{pd}}}

\def\fl{\operatorname{\mathsf{fl}}}

\def\db{\mathbf{D}_{\mathrm{b}}}
\def\ds{\mathbf{D}_{\mathrm{sg}}}

\def\M{\mathsf{e}}

\def\H{\mathsf{H}}
\def\K{\mathsf{K}}
\def\Ext{\mathsf{Ext}}
\def\syz{\mathsf{\Omega}}
\def\tr{\mathsf{Tr}}
\def\nf{\mathsf{NF}}
\def\v{\mathsf{V}}
\def\d{\mathsf{D}}
\def\s{\mathsf{S}}

\def\A{\mathcal{A}}
\def\X{\mathcal{X}}
\def\Y{\mathcal{Y}}
\def\Z{\mathcal{Z}}
\def\T{\mathcal{T}}
\def\C{\mathcal{C}}

\def\ZZ{\mathbb{Z}}
\begin{document}
\setlength{\baselineskip}{15pt}
\title[Reconstruction from Koszul homology]{Reconstruction from Koszul homology and applications to module and derived categories}
\author{Ryo Takahashi}
\address{Graduate School of Mathematics, Nagoya University, Furocho, Chikusaku, Nagoya, Aichi 464-8602, Japan}
\email{takahashi@math.nagoya-u.ac.jp}
\urladdr{http://www.math.nagoya-u.ac.jp/~takahashi/}
\thanks{2010 {\em Mathematics Subject Classification.} 13C60, 13D09, 18E30, 18E35}
\thanks{{\em Key words and phrases.} Koszul complex, Koszul homology, module category, derived category, singularity category, resolving subcategory, thick subcategory}
\thanks{The author was partially supported by JSPS Grant-in-Aid for Young Scientists (B) 22740008 and by JSPS Postdoctoral Fellowships for Research Abroad}
\begin{abstract}
Let $R$ be a commutative noetherian ring.
Let $M$ be a finitely generated $R$-module.
In this paper, we reconstruct $M$ from its Koszul homology with respect to a suitable sequence of elements of $R$ by taking direct summands, syzygies and extensions, and count the number of those operations.
Using this result, we consider generation and classification of certain subcategories of the category of finitely generated $R$-modules, its bounded derived category and the singularity category of $R$.
\end{abstract}
\maketitle
\section{Introduction}\label{intro}

For the past five decades, a lot of classification theorems of subcategories of abelian categories and triangulated categories have been given in ring theory, representation theory, algebraic geometry and algebraic topology; see, for instance, \cite{Ba,Ba2,BIK,crspd,FP,G,HS,H,K,KS,N0,St,stcm,crs,T} and the references therein.
Reconstruction of an object from its {\em support} in the spectrum of a suitable commutative ring plays a crucial role in the proofs of those theorems.

On the other hand, the notions of the dimensions of triangulated categories have been introduced by Bondal--Van den Bergh and Rouquier \cite{BV,R} and analogues for abelian categories by Dao--Takahashi \cite{radius,dim}.
These essentially indicate the number of {\em extensions} necessary to build all objects out of a single object.
There are many related studies; for example, see \cite{ddc,ABIM,AI2,BFK,BIKO,BC,Cd,KK,sing,Op,Or,R2,S,res}.

In this paper, we study reconstructing a given module from its Koszul homology and counting the number of necessary operations.
Our main result is the following theorem.

\begin{thm}\label{main}
Let $R$ be a commutative noetherian ring, and let $M$ be a finitely generated $R$-module.
Let $\xx=x_1,\dots,x_n$ be a sequence of elements of $R$ such that $M$ is locally free on $\d(\xx)$.
Then there exists a positive integer $k$ such that the Koszul complex $\K(\xx^k,M)$ is equivalent to a complex of finitely generated $R$-modules
$$
(0 \to N \to P_{n-1} \to \cdots \to P_0 \to 0),
$$
where $P_0,\dots,P_{n-1}$ are projective and $M$ is a direct summand of $N$.
In particular, $M$ can be built out of the Koszul homologies $\H_0(\xx^k,M),\dots,\H_n(\xx^k,M)$ by taking $n$ syzygies, $n$ extensions and $1$ direct summand.
\end{thm}

Note that since the free locus of a finitely generated $R$-module is an open subset of $\spec R$ in the Zariski topology, there exist many such sequences $\xx$ that satisfy the assumption of the theorem.
We shall prove a more general result in Theorem \ref{181627}.

Theorem \ref{main} has a lot of applications.
To state some of them, we fix notation.
Let $\mod R$ be the category of finitely generated $R$-modules and $\db(R)$ the bounded derived category of $\mod R$.
We denote by $\ds(R)$ the {\em singularity category} of $R$.
This category has been introduced and studied by Buchweitz \cite{B} in connection with Cohen--Macaulay modules over Gorenstein rings.
In recent years, it has been investigated by Orlov \cite{O1,O2,O3,O,O5} in relation to the Homological Mirror Symmetry Conjecture.

Let $\s(R)$ be the set of prime ideals $\p$ of $R$ such that $R_\p$ is not a field, and denote by $\sing R$ the singular locus of $R$.
Applying Theorem \ref{main}, we can prove the following result on classification of subcategories.

\begin{cor}\label{301200}
Let $R$ be a commutative noetherian ring.
\begin{enumerate}[\rm(1)]
\item
There is a one-to-one correspondence between:
\begin{itemize}
\item
the specialization-closed subsets of $\s(R)$,
\item
the resolving subcategories of $\mod R$ generated by a Serre subcategory of $\mod R$.
\end{itemize}
\item
There are one-to-one correspondences among:
\begin{itemize}
\item
the specialization-closed subsets of $\sing R$,
\item
the thick subcategories of $\db(R)$ generated by $R$ and a Serre subcategory of $\mod R$,
\item
the thick subcategories of $\ds(R)$ generated by a Serre subcategory of $\mod R$.
\end{itemize}
\end{enumerate}
\end{cor}

When $R$ is local, let $\mod^\circ(R)$ (respectively, $\db^\circ(R)$, $\ds^\circ(R)$) be the full subcategories of $\mod R$ (respectively, $\db(R)$, $\ds(R)$) consisting of modules (respectively, complexes) that are locally free (respectively, perfect, zero) on the punctured spectrum of $R$.
Applying Theorem \ref{main}, we can prove the following result on generation of subcategories.

\begin{cor}\label{011103}
Let $R$ be a commutative noetherian local ring of Krull dimension $d$ with residue field $k$.
\begin{enumerate}[\rm(1)]
\item
Every object in $\mod^\circ(R)$ is built out of a module of finite length by taking $d$ extensions in $\mod R$, up to finite direct sums, direct summands and syzygies.
\item
Every object in $\ds^\circ(R)$ is built out of a module of finite length by taking $d$ extensions in $\ds(R)$, up to finite direct sums, direct summands and shifts.
\end{enumerate}
In particular, one has that $\mod^\circ(R)$ is generated by $k$ as a resolving subcategory of $\mod R$, that $\db^\circ(R)$ is generated by $R$ and $k$ as a thick subcategory of $\db(R)$, and that $\ds^\circ(R)$ is generated by $k$ as a thick subcategory of $\ds(R)$.
\end{cor}

Corollary \ref{011103} yields variants of results shown by Schoutens \cite{S} and Takahashi \cite{res,stcm}.
It also recovers a result on isolated singularities given by Keller--Murfet--Van den Bergh \cite{KMV}.
Furthermore, utilizing it, one can show the following result.

\begin{cor}\label{301201}
Let $R$ be a commutative noetherian ring.
The following are equivalent for a resolving subcategory $\X$ of $\mod R$:
\begin{enumerate}[\rm(1)]
\item
$\X$ is generated by a Serre subcategory of $\mod R$;
\item
$\X$ is closed under tensor products and transposes.
\end{enumerate}
Hence there is a one-to-one correspondence between the specialization-closed subsets of $\s(R)$ and the resolving subcategories of $\mod R$ closed under tensor products and transposes.
\end{cor}

The last assertion of this corollary highly improves the main result of \cite{crs}.
Indeed, it removes the superfluous assumptions that $R$ is local and that $R$ is Cohen--Macaulay.

The organization of this paper is as follows.
In the next Section \ref{Bd} we prepare some fundamental notions.
In Section \ref{Rap} we state and prove the most general result in this paper, which includes Theorem \ref{main}.
In the final Section \ref{gene} we apply the results shown in the preceding section to find out the structure of certain subcategories, and give several results including Corollaries \ref{301200}, \ref{011103} and \ref{301201}.

\section{Basic definitions}\label{Bd}

This section is devoted to stating the definitions and basic properties of notions which we will {\em freely} use in the later sections.
We begin with our convention.

\begin{conv}
Throughout the present paper, let $R$ be a commutative noetherian ring with identity.
We assume that all $R$-modules are  finitely generated, that all $R$-complexes are homologically bounded, and that all subcategories of categories are full.
\end{conv}

In what follows, $\T$ and $\A$ denote a triangulated category and an abelian category with enough projective objects, respectively.

\begin{dfn}
(1) For a subcategory $\X$ of an additive category $\C$, the {\em additive closure} $\add_\C\X$ of $\X$ is defined to be the smallest subcategory of $\C$ containing $\X$ and closed under finite direct sums and direct summands.\\
(2) A {\em Serre subcategory} of $\A$ is defined to be a subcategory of $\A$ closed under subobjects, quotients and extensions.\\
(3) A {\em thick subcategory} of $\T$ is by definition a triangulated subcategory of $\T$ closed under direct summands.
The {\em thick closure} of a subcategory $\X$ of $\T$ is defined as the smallest thick subcategory of $\T$ containing $\X$, and denoted by $\thick_\T\X$ or simply by $\thick\X$.
When $\X$ consists of a single object $M$, we denote it by $\thick_\T M$ or $\thick M$.\\
(4) We denote by $\proj\A$ the subcategory of $\A$ consisting of projective objects.\\
(5) Let $P=(\cdots \xrightarrow{d_3} P_2 \xrightarrow{d_2} P_1 \xrightarrow{d_1} P_0 \to 0)$ be a projective resolution of $M\in\A$.
Then for each $n>0$ we define the {\em $n$-th syzygy} $\syz^nM$ of $M$ (with respect to $P$) as the image of $d_n$.
This is uniquely determined up to projective summands.\\
(6) We define a {\em resolving subcategory} of $\A$ as a subcategory of $\A$ containing $\proj\A$ and closed under direct summands, extensions and syzygies.
The {\em resolving closure} of a subcategory $\X$ of $\A$ is by definition the smallest resolving subcategory of $\A$ containing $\X$, and denoted by $\res_\A\X$ or simply by $\res\X$.
When $\X$ consists of a single object $M$, we denote it by $\res_\A M$ or $\res M$.\\
(7) Let $X,Y$ be complexes of objects of $\A$.\\
(a) A homomorphism $f:X\to Y$ of complexes is called a {\em quasi-isomorphism} if the induced map $\H_i(f):\H_i(X)\to\H_i(Y)$ on the $i$-th homologies is an isomorphism for all integers $i$.\\
(b) We say that $X$ is {\em equivalent} to $Y$ if there exists a sequence $X=X^0,X^1,\dots,X^n=Y$ of complexes having a quasi-isomorphism between $X^i$ and $X^{i+1}$ for all $0\le i\le n-1$.
Then we write $X\simeq Y$.
\end{dfn}

\begin{rem}
(1) A Serre subcategory is defined for an arbitrary abelian category.\\
(2) A resolving subcategory is usually defined as a subcategory containing the projective objects and closed under direct summands, extensions and kernels of epimorphisms.
This definition and ours are equivalent.\\
(3) Let $\X$ be a resolving subcategory of $\A$.
Let $M$ be an object of $\X$ and $n>0$ an integer.
The $n$-th syzygy of $M$ with respect to {\em some} projective resolution of $M$ is in $\X$ if and only if the $n$-th syzygy of $M$ with respect to {\em every} projective resolution of $M$ is in $\X$.
\end{rem}

We recall the notions of balls in $\T$ and $\A$ introduced in \cite{BV,radius,R}.

\begin{dfn}
(1)(a) For a subcategory $\X$ of $\T$ we denote by $\langle\X\rangle$ the smallest subcategory of $\T$ containing $\X$ that is closed under finite direct sums, direct summands and shifts, i.e., $\langle\X\rangle=\add_\T\{\,X[i]\mid i\in\ZZ,\,X\in\X\,\}$.
When $\X$ consists of a single object $M$, we simply denote it by $\langle M\rangle$.\\
(b) For subcategories $\X,\Y$ of $\T$ we denote by $\X\ast\Y$ the subcategory of $\T$ consisting of objects $M$ which fits into an exact triangle $X \to M \to Y \rightsquigarrow$ in $\T$ with $X\in\X$ and $Y\in\Y$.
We set $\X\diamond\Y=\langle\langle\X\rangle\ast\langle\Y\rangle\rangle$.\\
(c) Let $\C$ be a subcategory of $\T$.
We define the {\it ball of radius $r$ centered at $\C$} as
$$
\langle\C\rangle_r=
\begin{cases}
\langle\C\rangle & (r=1),\\
\langle\C\rangle_{r-1}\diamond\C=\langle\langle\C\rangle_{r-1}\ast\langle\C\rangle\rangle & (r\ge2).
\end{cases}
$$
If $\C$ consists of a single object $M$, then we simply denote it by $\langle M\rangle_r$.
We write $\langle\C\rangle_r^{\T}$ when we should specify that $\T$ is the ground category where the ball is defined.\\
(2)(a) For a subcategory $\X$ of $\A$ we denote by $[\X]$ the smallest subcategory of $\A$ containing $\proj\A$ and $\X$ that is closed under finite direct sums, direct summands and syzygies, i.e., $[\X]=\add_\A(\proj\A\cup\{\,\syz^iX\mid i\ge0,\,X\in\X\,\})$.
When $\X$ consists of a single object $M$, we simply denote it by $[M]$.\\
(b) For subcategories $\X,\Y$ of $\A$ we denote by $\X\circ\Y$ the subcategory of $\A$ consisting of objects $M$ which fits into an exact sequence $0 \to X \to M \to Y \to 0$ in $\A$ with $X\in\X$ and $Y\in\Y$.
We set $\X\bullet\Y=[[\X]\circ[\Y]]$.\\
(c) Let $\C$ be a subcategory of $\A$.
We define the {\it ball of radius $r$ centered at $\C$} as
$$
[\C]_r=
\begin{cases}
[\C] & (r=1),\\
[\C]_{r-1}\bullet\C=[[\C]_{r-1}\circ[\C]] & (r\ge2).
\end{cases}
$$
If $\C$ consists of a single object $M$, then we simply denote it by $[M]_r$.
We write $[\C]_r^\A$ when we should specify that $\A$ is the ground category where the ball is defined.
\end{dfn}

\begin{rem}\cite{BV,radius,R}
(1) Let $\X,\Y,\Z,\C$ be subcategories of $\T$.\\
(a) An object $M\in\T$ belongs to $\X\diamond\Y$ if and only if there is an exact triangle $X \to Z \to Y \rightsquigarrow$ with $X\in\langle\X\rangle$ and $Y\in\langle\Y\rangle$ such that $M$ is a direct summand of $Z$.\\
(b) One has $(\X\diamond\Y)\diamond\Z=\X\diamond(\Y\diamond\Z)$ and $\langle\C\rangle_a\diamond\langle\C\rangle_b=\langle\C\rangle_{a+b}$ for all $a,b>0$.\\
(2) Let $\X,\Y,\Z,\C$ be subcategories of $\A$.\\
(a) An object $M\in\A$ belongs to $\X\bullet\Y$ if and only if there is an exact sequence $0 \to X \to Z \to Y \to 0$ with $X\in[\X]$ and $Y\in[\Y]$ such that $M$ is a direct summand of $Z$.\\
(b) One has $(\X\bullet\Y)\bullet\Z=\X\bullet(\Y\bullet\Z)$ and $[\C]_a\bullet[\C]_b=[\C]_{a+b}$ for all $a,b>0$.
\end{rem}


\begin{dfn}
An $R$-complex is called {\em perfect} if it is a bounded complex of projective $R$-modules.
The {\em singularity category} $\ds(R)$ of $R$ is defined as the Verdier quotient of $\db(R)$ by the perfect complexes.
For the definition of a Verdier quotient, we refer to \cite[Remark 2.1.9]{N}.
Whenever we discuss the singularity category $\ds(R)$, we identify each object or subcategory of $\mod R$ with its image in $\ds(R)$ by the composition of the canonical functors $\mod R\to\db(R)\to\ds(R)$.
\end{dfn}

\begin{rem}\cite[Lemma 2.4]{sing}
(1) For all $X\in\db(R)$ there exists an exact triangle $P \to X \to M[n] \rightsquigarrow$ in $\db(R)$ such that $P$ is a perfect complex, $M$ is a module and $n$ is an integer.
In particular, $X\cong M[n]$ in $\ds(R)$.\\
(2) For every $M\in\mod R$ and every $n\ge0$ there is an isomorphism $M\cong\syz^nM[n]$ in $\ds(R)$.
Hence, for a subcategory $\C$ of $\mod R$ and an integer $k>0$, each module in ${[\C]}_k^{\mod R}$ belongs to ${\langle\C\rangle}_k^{\ds(R)}$.
\end{rem}

We introduce subcategories which will be investigated in Section \ref{gene}.

\begin{dfn}
Let $\Phi$ be a subset of $\spec R$.
Set $\Phi^{\rm c}=\spec R\setminus\Phi$.
We denote by $\M^\Phi(R)$ (respectively, $\mod^\Phi(R)$) the subcategory of $\mod R$ consisting of $R$-modules $M$ such that $M_\p=0$ (respectively, $M_\p$ is $R_\p$-free) for all $\p\in\Phi^{\rm c}$.
Also, $\db^\Phi(R)$ (respectively, $\ds^\Phi(R)$) denotes the subcategory of $\db(R)$ (respectively, $\ds(R)$) consisting of $R$-complexes $X$ such that $X_\p$ isomorphic to a perfect $R_\p$-complex in $\db(R_\p)$ (respectively, $X_\p\cong0$ in $\ds(R_\p)$) for all $\p\in\Phi^{\rm c}$.
We have that $\M^\Phi(R)$ is a Serre subcategory of $\mod R$, that $\mod^\Phi(R)$ is a resolving subcategory of $\mod R$, and that $\db^\Phi(R), \ds^\Phi(R)$ are thick subcategories of $\db(R),\ds(R)$ respectively.
\end{dfn}


\begin{dfn}
(1) For an $R$-module $M$ we denote by $\nf(M)$ the {\em nonfree locus} of $M$, that is, the set of prime ideals $\p$ of $R$ such that the $R_\p$-module $M_\p$ is nonfree.
As is well-known, $\nf(M)$ is a closed subset of $\spec R$ in the Zariski topology.\\
(2) For an $R$-complex $M$ we denote by $\ipd(M)$ the {\em infinite projective dimension locus} of $M$, that is, the set of prime ideals $\p$ of $R$ such that the $R_\p$-complex $M_\p$ has infinite projective dimension.\\
(3) For a subcategory $\X$ of $\mod R$ we set $\supp\X=\bigcup_{M\in\X}\supp M$ and $\nf(\X)=\bigcup_{M\in\X}\nf(M)$.\\
(4) For a subcategory $\X$ of $\db(R)$ we set $\ipd(\X)=\bigcup_{M\in\X}\ipd(M)$.\\
(5) For a subcategory $\X$ of $\ds(R)$ we set $\ssupp(\X)=\bigcup_{M\in\X}\ipd(M)$.
\end{dfn}


\begin{dfn}
(1) Let $M$ be an $R$-module.\\
(a) Let $\xx$ be a sequence of elements of $R$.
Then $\K(\xx,M)$ denotes the {\em Koszul complex} of $M$ with respect to $\xx$.
We call $\H_i(\xx,M):=\H_i(\K(\xx,M))$ the $i$-th {\em Koszul homology} ($i\in\ZZ$) and $\H(\xx,M):=\bigoplus_{i\in\ZZ}\H_i(\xx,M)$ the {\em Koszul homology} of $M$ with respect to $\xx$.\\
(b) Let $P_1 \xrightarrow{d} P_0 \to M \to 0$ be a projective presentation of $M$.
Then the cokernel of the $R$-dual map of $d$ is called the {\em transpose} of $M$ and denoted by $\tr M$.
This is uniquely determined up to projective summands.\\
(3) A subset $\Phi$ of $\spec R$ is called {\em specialization-closed} if $\v(\p)\subseteq\Phi$ for all $\p\in\Phi$.
This is nothing but a union of closed subsets of $\spec R$ in the Zariski topology.\\
(4) We denote by $\sing R$ the {\em singular locus} of $R$, namely, the set of prime ideals $\p$ of $R$ such that $R_\p$ is not a regular local ring.\\
(5) A local ring $R$ with maximal ideal $\m$ is called an {\em isolated singularity} if $\sing R\subseteq\{\m\}$.
\end{dfn}

\section{Reconstruction from Koszul homology}\label{Rap}

In this section, we consider reconstructing a given module from its Koszul homology by taking direct summands, extensions and syzygies.
We start by stating and proving the most general result in this paper; actually, almost all of the other results given in this paper are deduced from this.

\begin{thm}\label{181627}
Let $M$ be an $R$-module.
Let $\xx=x_1,\dots,x_n$ be a sequence of elements of $R$ such that $x_p\Ext_R^q(M,\syz^rM)=0$ for all $1\le p\le n$ and $1\le q,r\le p$.
Let $P$ be a projective resolution of $M$.
Then $\K(\xx,M)$ is equivalent to a complex
$$
X=(0 \to X_n \to X_{n-1} \to \cdots \to X_1 \to X_0 \to 0)
$$
such that $X_i=\bigoplus_{j=0}^i{P_j}^{\oplus\binom{n}{i-j}}$ for each $0\le i\le n-1$ and $X_n=\bigoplus_{j=0}^n(\syz^jM)^{\oplus\binom{n}{j}}$.
\end{thm}

\begin{proof}
We prove the theorem by induction on $n$.
Let us first consider the case where $n=1$.
Multiplication by $x_1$ makes a pullback diagram:
$$
\begin{CD}
\phantom{x_1}\sigma:\quad @. 0 @>>> \syz M @>>> P_0 @>>> M @>>> 0\phantom{.} \\
@. @. @| @AAA @A{x_1}AA \\
x_1\sigma:\quad @. 0 @>>> \syz M @>>> N @>>> M @>>> 0.
\end{CD}
$$
Since $x_1\Ext_R^1(M,\syz M)=0$, we see that the exact sequence $x_1\sigma$ splits and get an isomorphism $N\cong\syz M\oplus M$.
Thus we obtain a short exact sequence of complexes
$$
0 \to W \to X \to \K(x_1,M) \to 0,
$$
where $W=(0 \to \syz M \xrightarrow{=} \syz M \to 0)$ and $X=(0 \to \syz M\oplus M \to P_0 \to 0)$.
As $W$ is acyclic, $\K(x_1,M)$ is equivalent to $X$.

Next, we assume $n\ge2$.
The induction hypothesis implies that $\K(x_1,\dots,x_{n-1},M)$ is equivalent to a complex
$$
Y=(0 \to Y_{n-1} \xrightarrow{f} Y_{n-2} \to \cdots \to Y_1 \to Y_0 \to 0)
$$
with $Y_i=\bigoplus_{j=0}^i{P_j}^{\oplus\binom{n-1}{i-j}}$ for $0\le i\le n-2$ and $Y_{n-1}=\bigoplus_{j=0}^{n-1}(\syz^jM)^{\oplus\binom{n-1}{j}}$.
In general, taking a tensor product with a perfect complex preserves equivalence of complexes (cf. \cite[(A.4.1)]{C}).
Hence we have
\begin{align*}
&\K(\xx,M)
=\K(x_1,\dots,x_{n-1},M)\otimes_R\K(x_n,R)
\simeq Y\otimes_R\K(x_n,R)\\
&=(0 \to Y_{n-1} \xrightarrow{g} Y_{n-1}\oplus Y_{n-2} \xrightarrow{d_{n-1}} Y_{n-2}\oplus Y_{n-3} \xrightarrow{d_{n-2}} \cdots \xrightarrow{d_2} Y_1\oplus Y_0 \xrightarrow{d_1} Y_0 \to 0)=:Z,
\end{align*}
where $g=\binom{(-1)^{n-1}x_n}{f}$.
Note that there is an exact sequence $0 \to \syz Y_{n-1} \to Q \xrightarrow{\pi} Y_{n-1} \to 0$ with $Q=\bigoplus_{j=0}^{n-1}{P_j}^{\oplus\binom{n-1}{j}}$.
Consider the pullback diagram
$$
\begin{CD}
\phantom{g^\ast()}\tau:\quad @. 0 @>>> \syz Y_{n-1} @>>> Q\oplus Y_{n-2} @>{h}>> Y_{n-1}\oplus Y_{n-2} @>>> 0\phantom{,} \\
@. @. @| @AAA @A{g}AA \\
g^\ast(\tau):\quad @. 0 @>>> \syz Y_{n-1} @>>> L @>>> Y_{n-1} @>>> 0,
\end{CD}
$$
where $h=\left(\begin{smallmatrix}
\pi & 0 \\
0 & 1
\end{smallmatrix}\right)$ and $g^\ast=\Ext_R^1(g,\syz Y_{n-1})$.
As $Y_{n-2}$ is projective, the map $g^\ast$ can be identified with the multiplication map $\Ext_R^1(Y_{n-1},\syz Y_{n-1})\xrightarrow{(-1)^{n-1}x_n}\Ext_R^1(Y_{n-1},\syz Y_{n-1})$.
There are isomorphisms
\begin{align*}
\Ext_R^1(Y_{n-1},\syz Y_{n-1})
&\textstyle\cong\bigoplus_{j,k=0}^{n-1}\Ext_R^1(\syz^jM,\syz(\syz^kM))^{\oplus\left(\binom{n-1}{j}+\binom{n-1}{k}\right)}\\
&\textstyle\cong\bigoplus_{j,k=0}^{n-1}\Ext_R^{j+1}(M,\syz^{k+1}M)^{\oplus\left(\binom{n-1}{j}+\binom{n-1}{k}\right)},
\end{align*}
and hence $x_n$ annihilates $\Ext_R^1(Y_{n-1},\syz Y_{n-1})$.
Therefore $g^\ast(\tau)$ is a split exact sequence, and we obtain a commutative diagram
$$
\begin{CD}
0 @>>> \syz Y_{n-1} @>>> Q\oplus Y_{n-2} @>h>> Y_{n-1}\oplus Y_{n-2} @>>> 0\phantom{,} \\
@. @| @AlAA @AgAA \\
0 @>>> \syz Y_{n-1} @>>> \syz Y_{n-1}\oplus Y_{n-1} @>>> Y_{n-1} @>>> 0,
\end{CD}
$$
with exact rows.
We observe that the complex $Z$ is equivalent to the complex
$$
X=(0 \to \syz Y_{n-1}\oplus Y_{n-1} \xrightarrow{l} Q\oplus Y_{n-2} \xrightarrow{d_{n-1}h} Y_{n-2}\oplus Y_{n-3} \xrightarrow{d_{n-2}} \cdots \xrightarrow{d_2} Y_1\oplus Y_0 \xrightarrow{d_1} Y_0 \to 0).
$$
There are equalities $\syz Y_{n-1}\oplus Y_{n-1}=\bigoplus_{j=0}^n(\syz^jM)^{\oplus\binom{n}{j}}$, $Q\oplus Y_{n-2}=\bigoplus_{j=0}^{n-1}{P_j}^{\oplus\binom{n}{(n-1)-j}}$, $Y_i\oplus Y_{i-1}=\bigoplus_{j=0}^i{P_j}^{\oplus\binom{n}{i-j}}$ for $1\le i\le n-2$ and $Y_0=P_0$.
Thus we are done.
\end{proof}

Using Theorem \ref{181627}, we obtain the following corollary.

\begin{cor}\label{241142}
Let $M$ and $\xx$ be as in Theorem \ref{181627}.
\begin{enumerate}[\rm(1)]
\item
If $\xx$ is a regular sequence on $M$, then $\syz^n(M/\xx M)\cong\bigoplus_{k=0}^n(\syz^kM)^{\oplus\binom{n}{k}}$ in $\mod R$.
\item
For each $1\le i\le n$ there exists an exact sequence of $R$-modules
$$
0 \to \H_i(\xx,M) \to E_i \to \syz E_{i-1} \to 0
$$
with $E_0=\H_0(\xx,M)$ such that $M$ is a direct summand of $E_n$.
Hence $M$ is built out of $\H_0(\xx,M),\dots,\H_n(\xx,M)$ by taking $n$ syzygies, $n$ extensions and $1$ direct summand.
In particular, $M$ belongs to the ball $[\H(\xx,M)]_{n+1}^{\mod R}$.
\item
There is an exact triangle
$$
\textstyle
F \to \K(\xx,M) \to \bigoplus_{k=0}^n(\syz^kM)^{\oplus\binom{n}{k}}[n] \rightsquigarrow
$$
in $\db(R)$, where $F=(0 \to F_{n-1} \to \cdots \to F_0 \to 0)$ is a perfect complex.
\item
The module $M$ belongs to the ball ${\langle R\oplus\K(\xx,M)\rangle}_{n+1}^{\db(R)}$.
\item
One has $\K(\xx,M)\cong\bigoplus_{k=0}^nM^{\oplus\binom{n}{k}}[k]$ in $\ds(R)$.
In particular, $M$ is a direct summand of $\K(\xx,M)$ in $\ds(R)$.
\end{enumerate}
\end{cor}

\begin{proof}
We use the notation of Theorem \ref{181627} and its assertion.

(1) Since $\xx$ is regular on $M$, we have an equivalence $\K(\xx,M)\simeq M/\xx M$.
There is an exact sequence
$$
0 \to X_n \to X_{n-1} \to \cdots \to X_0 \to M/\xx M \to 0
$$
of $R$-modules.
As $X_n=\bigoplus_{j=0}^n(\syz^jM)^{\oplus\binom{n}{j}}$ and $X_i$ is projective for all $0\le i\le n-1$, the module $X_n$ is the $n$-th syzygy of $M/\xx M$ as an $R$-module.

(2) 
For each $0\le i\le n$ take a truncation $X^i=(0 \to X_n \to \cdots \to X_{i+1} \to X_i \to 0)$ of $X$ with $(X^i)_j=X_{i+j}$ for $0\le j\le n$.
Then there is a short exact sequence
$$
0 \to X_{i-1} \to X^{i-1} \to X^i[1] \to 0
$$
of complexes for each $1\le i\le n$.
The homology long exact sequence gives an exact sequence $0 \to \H_1(X^{i-1}) \to \H_0(X^i) \to X_{i-1} \to \H_0(X^{i-1}) \to 0$ of modules.
As $X_{i-1}$ is projective, we have an exact sequence
$$
0 \to \H_1(X^{i-1}) \to \H_0(X^i) \to \syz\H_0(X^{i-1}) \to 0
$$
for all $1\le i\le n$.
Notice $\H_1(X^{i-1})=\H_i(\xx,M)$, $\H_0(X^0)=\H_0(\xx,M)$ and $\H_0(X^n)=X_n$.
Setting $E_i=\H_0(X^i)$ for $0\le i\le n$, we obtain desired exact sequences.

(3) Truncating the complex $X$ provides such an exact triangle.

(4) Decomposing $F$ into short exact sequences of complexes, we observe that $F$ is in ${\langle R\rangle}_n^{\db(R)}$.
As $M$ is a direct summand of $\bigoplus_{k=0}^n(\syz^kM)^{\oplus\binom{n}{k}}$, the assertion follows from (3).

(5) By (3) we have an isomorphism $\K(\xx,M)\cong\bigoplus_{k=0}^n(\syz^kM)^{\oplus\binom{n}{k}}[n]$ in $\ds(R)$.
Since $M\cong\syz^kM[k]$ in $\ds(R)$, we are done.
\end{proof}

\begin{rem}
(1) Corollary \ref{241142}(1) is a refinement of \cite[Proposition 2.2]{stcm}, which shows the same conclusion under the additional assumption that $\xx$ is a regular sequence on $R$ annihilating more Ext modules.\\
(2) Corollary \ref{241142}(5) can also be shown by using the proof of \cite[Proposition 2.3]{sing}.
It also implies that $M$ belongs to ${\langle R\oplus\K(\xx,M)\rangle}_m^{\db(R)}$ for {\em some} integer $m>0$.
However, it cannot determine how big/small $m$ is, while Corollary \ref{241142}(4) can.
\end{rem}

We are interested in existence of a sequence $\xx$ as in Theorem \ref{181627}.
The lemma below guarantees that such a sequence always exists.
Moreover, one can make such a sequence as a power of an arbitrary sequence whose defining closed subset covers the nonfree locus.

\begin{lem}\label{181842}
Let $M$ be an $R$-module.
Let $\xx=x_1,\dots,x_n$ be a sequence of elements of $R$ with $\nf(M)\subseteq\v(\xx)$.
Then there exists an integer $k>0$ such that the sequence $\xx^k=x_1^k,\dots,x_n^k$ annihilates $\Ext_R^i(M,N)$ for all $i>0$ and all $N\in\mod R$.
\end{lem}

\begin{proof}
Let $I$ be an ideal of $R$ with $\nf(M)=\v(I)$.
Then by \cite[Remark 5.2(1)]{dim} there exists an integer $p>0$ such that $I^p\Ext_R^i(M,N)=0$ for all $i>0$ and all $N\in\mod R$.
By assumption, we have $(\xx^q)\subseteq I$ for some $q>0$.
Setting $k=pq$ completes the proof.
\end{proof}

Combining Theorem \ref{181627}, Corollary \ref{241142}(2) and Lemma \ref{181842}, we immediately obtain the following result, which includes Theorem \ref{main}.

\begin{cor}\label{231851}
Let $M$ be an $R$-module.
Let $\xx=x_1,\dots,x_n$ be a sequence of elements of $R$ with $\nf(M)\subseteq\v(\xx)$.
Then there exists an integer $k>0$ such that $\K(\xx^k,M)$ is equivalent to a complex
$$
(0 \to N \to P_{n-1} \to \cdots \to P_0 \to 0),
$$
where each $P_i$ is projective and $M$ is a direct summand of $N$.
Hence, $M$ is built out of $\H_0(\xx^k,M),\dots,\H_n(\xx^k,M)$ by taking $n$ syzygies, $n$ extensions and $1$ direct summand.
In particular, $M$ is in $[\H(\xx^k,M)]_{n+1}^{\mod R}$.
\end{cor}

\section{Generation of subcategories}\label{gene}

In this section, we apply our results obtained in the previous section to investigate generation of subcategories.
To be precise, for a subset $\Phi$ of $\spec R$ we analyze the structure of the subcategories $\mod^\Phi(R)$, $\db^\Phi(R)$ and $\ds^\Phi(R)$.
We also consider classification of these subcategories.

First of all, we want to make a generator of $\mod^\Phi(R)$ as a resolving subcategory of $\mod R$ and generators of $\db^\Phi(R), \ds^\Phi(R)$ as thick subcategories of $\db(R),\ds(R)$.
In fact, $\M^\Phi(R)$ gives generators of these three subcategories:

\begin{thm}\label{190906}
Let $\Phi$ be a subset of $\spec R$.
Then one has equalities
\begin{align}
\mod^\Phi(R)&=\res_{\mod R}(\M^\Phi(R)),\\
\db^\Phi(R)&=\thick_{\db(R)}(\{R\}\cup\M^\Phi(R)),\\
\ds^\Phi(R)&=\thick_{\ds(R)}(\M^\Phi(R)).
\end{align}
\end{thm}

\begin{proof}
(1) It is obvious that $\M^\Phi(R)$ is contained in $\mod^\Phi(R)$, and hence so is its resolving closure.
To show the opposite inclusion, let $M$ be an object of $\mod^\Phi(R)$.
Then by definition $\nf(M)$ is contained in $\Phi$.
It is seen from Corollary \ref{231851} that there is a sequence $\xx=x_1,\dots,x_n$ of elements of $R$ with $\nf(M)=\v(\xx)$ such that $M$ belongs to $\res_{\mod R}\H(\xx,M)$.
Since $\H(\xx,M)$ is annihilated by $\xx$, we have
$$
\supp\H(\xx,M)\subseteq\v(\xx)=\nf(M)\subseteq\Phi,
$$
which shows $\H(\xx,M)\in\M^{\Phi}(R)$.
Consequently, $M$ is in $\res_{\mod R}(\M^\Phi(R))$.

(2) Clearly, $\db^\Phi(R)$ contains $R$ and $\M^\Phi(R)$, and the thick closure of $\{R\}\cup\M^\Phi(R)$.
Let $X$ be an object of $\db^\Phi(R)$.
Then there is an exact triangle
$$
P \to X \to M[n] \rightsquigarrow
$$
in $\db(R)$ such that $P$ is a perfect $R$-complex, $M$ is an $R$-module and $n$ is an integer.
We use the {\em large restricted flat dimension}
$$
\Rfd_RM=\sup_{\p\in\spec R}\{\depth R_\p-\depth_{R_\p}M_\p\}
$$
of $M$.
By \cite[Theorem 1.1]{AIL} this is finite.
Put $r=\Rfd_RM$.
Let $\p$ be a prime ideal in $\Phi^{\rm c}$.
Localizing the above exact triangle at $\p$, we see that the $R_\p$-module $M_\p$ has finite projective dimension.
Hence $\pd_{R_\p}M_\p=\depth R_\p-\depth_{R_\p}M_\p\le r$.
Setting $N=\syz^rM$, we observe that $N$ belongs to $\mod^\Phi(R)$, hence to $\res_{\mod R}(\M^\Phi(R))$ by (1).
Therefore $N$ is in $\thick_{\db(R)}(\{R\}\cup\M^\Phi(R))$, and so is $M$.
As $P\in\thick_{\db(R)}R$, the object $X$ belongs to $\thick_{\db(R)}(\{R\}\cup\M^\Phi(R))$ by the above exact triangle.

(3) This equality is obtained by using (2).
\end{proof}

One can describe the structure of $\M^\Phi(R)$ in more detail, which makes more visible representations of $\mod^\Phi(R)$, $\db^\Phi(R)$ and $\ds^\Phi(R)$.

\begin{cor}\label{040950}
Let $\Phi$ be a subset of $\spec R$.
Then $\M^\Phi(R)$ is the smallest subcategory of $\mod R$ containing $R/\p$ for all $\p\in\Phi^{\rm sp}$ and closed under extensions.
Here $\Phi^{\rm sp}$ denotes the largest specialization-closed subset of $\spec R$ contained in $\Phi$.
Hence
\begin{align*}
\mod^\Phi(R)&=\res_{\mod R}\{\,R/\p\mid\p\in\Phi^{\rm sp}\,\},\\
\db^\Phi(R)&=\thick_{\db(R)}\{\,R,R/\p\mid\p\in\Phi^{\rm sp}\,\},\\
\ds^\Phi(R)&=\thick_{\ds(R)}\{\,R/\p\mid\p\in\Phi^{\rm sp}\,\}.
\end{align*}
\end{cor}

\begin{proof}
The last assertion follows from Theorem \ref{190906}.

We claim that $\Phi^{\rm sp}=\supp(\M^\Phi(R))$ holds.
Indeed, it is evident that $\supp(\M^\Phi(R))$ is a specialization-closed subset of $\spec R$ contained in $\Phi$.
Let $\Psi$ be a specialization-closed subset of $\spec R$ contained in $\Phi$.
Then we have $\M^\Psi(R)\subseteq\M^\Phi(R)$, and hence $\Psi=\supp(\M^\Psi(R))\subseteq\supp(\M^\Phi(R))$.
Thus the claim holds.

Let $\X$ be the smallest subcategory of $\mod R$ containing $R/\p$ for all $\p\in\Phi^{\rm sp}$ and closed under extensions.
First, let $\p$ be a prime ideal in $\Phi^{\rm sp}$.
As $\Phi^{\rm sp}$ is specialization-closed, we have $\supp(R/\p)=\v(\p)\subseteq\Phi^{\rm sp}\subseteq\Phi$, whence $R/\p$ belongs to $\M^\Phi(R)$.
Since $\M^\Phi(R)$ is closed under extensions, $\M^\Phi(R)$ contains $\X$.
Next, let $M$ be a module in $\M^\Phi(R)$.
Take a filtration
$$
M=M_0\supsetneq M_1\supsetneq\cdots\supsetneq M_n=0
$$
of submodules of $M$ such that $M_{i-1}/M_i\cong R/\p_i$ with $\p_i\in\spec R$ for each $1\le i\le n$.
Then $\p_i$ is in $\supp M$, and so in $\supp(\M^\Phi(R))$.
By the claim, we have $\p_i\in\Phi^{\rm sp}$ for all $1\le i\le n$.
Decomposing the above filtration into short exact sequences, we see that $M$ is in $\X$.
Therefore $\X$ contains $\M^\Phi(R)$, and the proof is completed.
\end{proof}

Corollary \ref{040950} immediately gives the following, which includes part of Corollary \ref{011103}.
Note that the objects of $\mod^{\{\m\}}(R)$ are the $R$-modules that are locally free on the punctured spectrum of $R$.

\begin{cor}\label{191104}
\begin{enumerate}[\rm(1)]
\item
The equalities $\db(R)=\thick\{\,R,R/\p\mid\p\in\sing R\,\}$ and $\ds(R)=\thick\{\,R/\p\mid\p\in\sing R\,\}$ hold.
\item
Suppose that $R$ is a local ring with maximal ideal $\m$ and residue field $k$.
Then $\mod^{\{\m\}}(R)=\res(k)$, $\db^{\{\m\}}(R)=\thick(R\oplus k)$ and $\ds^{\{\m\}}(R)=\thick(k)$.
\end{enumerate}
\end{cor}

\begin{rem}
(1) Corollary \ref{191104}(1) can also be shown by using \cite[Theorem VI.8]{S}.\\
(2) Similar results to Corollary \ref{191104}(2) are obtained in \cite[Theorem 2.4]{stcm} and \cite{AI}.
\end{rem}

As a common consequence of the two assertions of Corollary \ref{191104}, one can recover \cite[Proposition A.2]{KMV}:

\begin{cor}\label{192102}
Let $R$ be an isolated singularity with residue field $k$.
Then $\db(R)=\thick(R\oplus k)$ and $\ds(R)=\thick(k)$.
\end{cor}

Next, we make a closer investigation on the inner structure of subcategories.
In fact, we can refine the assertions as to $\mod^{\{\m\}}(R)$ and $\ds^{\{\m\}}(R)$ in Corollary \ref{191104}(2) in terms of balls in the abelian category $\mod R$ and the triangulated category $\ds(R)$.
Denote by $\fl(R)$ the subcategory of $\mod R$ consisting of modules of finite length.
The following theorem holds, which is the main part of Corollary \ref{011103}.

\begin{thm}\label{231924}
Let $R$ be a $d$-dimensional local ring with maximal ideal $\m$.
Then there are equalities $\mod^{\{\m\}}(R)=[\fl(R)]_{d+1}^{\mod R}$ and $\ds^{\{\m\}}(R)={\langle\fl(R)\rangle}_{d+1}^{\ds(R)}$.
\end{thm}

\begin{proof}
(1) Let us show the first equality.
It clearly holds when $d=0$, so we assume $d>0$.
Let $M$ be an $R$-module in $\mod^{\{\m\}}(R)$.
Take any system of parameters $\xx=x_1,\dots,x_d$ of $R$.
As $M$ is in $\mod^{\{\m\}}(R)$, we have $\nf(M)\subseteq\{\m\}=\v(\xx)$.
Corollary \ref{231851} implies that $M$ belongs to $[\H(\xx^k,M)]_{d+1}$ for some $k>0$.
Since the $R$-module $\H(\xx^k,M)$ is annihilated by the $\m$-primary ideal $(\xx^k)$, it has finite length.
Thus we obtain $M\in[\fl(R)]_{d+1}$, and the first equality follows.

(2) We prove the second equality.
Let $X$ be an $R$-complex in $\ds^{\{\m\}}(R)$.
Note that $X\cong\syz^d M[n]$ in $\ds(R)$ for some $R$-module $M$ and some integer $n$.
By the Auslander--Buchsbaum formula, we see that $\syz^dM$ belongs to $\mod^{\{\m\}}(R)=[\fl(R)]_{d+1}$.
Now the second equality follows from the first one.
\end{proof}

Here is an immediate consequence of Theorem \ref{231924}.

\begin{cor}\label{020945}
Let $R$ be a $d$-dimensional isolated singularity.
Then $\ds(R)={\langle\fl(R)\rangle}_{d+1}$.
\end{cor}

\begin{rem}
(1) Rewording the second equality in Theorem \ref{231924} by the terminology introduced in \cite{ddcm}, one has the following inequality:
$$
\fl(R)\operatorname{\text{-tri.dim}}\ds^{\{\m\}}(R)\le\dim R.
$$
(2) The result \cite[Theorem A]{res} constructs {\em some} object in $\mod^{\{\m\}}(R)$ from {\em every} object in $\mod R$ and counts the number of necessary operations (containing syzygies).
In contrast to this, Theorem \ref{231924} constructs {\em every} object in $\mod^{\{\m\}}(R)$ from {\em some} object in $\fl(R)$ and counts the number of necessary operations.\\
(3) Similar equalities to the first equality in Theorem \ref{231924} are given for $\mod R$ in \cite[Theorem VI.8]{S} and \cite[Theorem 2]{BC}, but these are different from ours in respect of how to count operations.
The biggest difference is that neither of those two results counts the number of necessary extensions.\\
(4) In the case where $R$ is Cohen--Macaulay, Corollary \ref{020945} also follows from \cite[(4.5.1)]{ddcm}, because every maximal Cohen--Macaulay $R$-module is a direct summand of the $d$-th syzygy of some module of finite length by \cite[Proposition 2.2]{stcm}.
\end{rem}

Finally, we are interested in classifying resolving and thick subcategories by using $\mod^\Phi(R)$, $\db^\Phi(R)$ and $\ds^\Phi(R)$.
For this purpose, we prepare a lemma:

\begin{lem}\label{301519}
\begin{enumerate}[\rm(1)]
\item
The assignments $\X\mapsto\supp\X$ and $\Phi\mapsto \M^\Phi(R)$ make a one-to-one correspondence between the Serre subcategories of $\mod R$ and the specialization-closed subsets of $\spec R$.
\item
Let $\Phi$ be a specialization-closed subset of $\spec R$.
Then $\nf(\mod^\Phi(R))=\Phi\cap\s(R)$ and $\ipd(\db^\Phi(R))=\ssupp(\ds^\Phi(R))=\Phi\cap\sing R$.
\end{enumerate}
\end{lem}

\begin{proof}
(1) This is nothing but Gabriel's classification theorem of Serre subcategories \cite{G}.

(2) Let $\p\in\Phi$.
Then $\ipd(R/\p)\subseteq\nf(R/\p)\subseteq\v(\p)\subseteq\Phi$.
Hence $R/\p$ belongs to $\mod^\Phi(R)$, $\db^\Phi(R)$ and $\ds^\Phi(R)$.
If $\p$ is in $\s(R)$ (respectively, $\sing R$), then $\p$ is in $\nf(R/\p)$ (respectively, $\ipd(R/\p)$).
The assertion now follows.
\end{proof}

We can obtain the following theorem, which includes Corollary \ref{301200}.

\begin{thm}\label{301626}
\begin{enumerate}[\rm(1)]
\item
The assignment $\Phi\mapsto\mod^\Phi(R)$ makes a bijection from the set of specialization-closed subsets of $\spec R$ contained in $\s(R)$ to the set of resolving closures $\res_{\mod R}\X$, where $\X$ runs through the Serre subcategories of $\mod R$.
\item
The assignment $\Phi\mapsto\db^\Phi(R)$ makes a bijection from the set of specialization-closed subsets of $\spec R$ contained in $\sing R$ to the set of thick closures $\thick_{\db(R)}(\{R\}\cup\X)$, where $\X$ runs through the Serre subcategories of $\mod R$.
\item
The assignment $\Phi\mapsto\ds^\Phi(R)$ makes a bijection from the set of specialization-closed subsets of $\spec R$ contained in $\sing R$ to the set of thick closures $\thick_{\ds(R)}\X$, where $\X$ runs through the Serre subcategories of $\mod R$.
\end{enumerate}
\end{thm}

\begin{proof}
In view of Theorem \ref{190906}, the three assignments make well-defined maps, and they are injective by Lemma \ref{301519}(2).
Thus it only remains to show that they are surjective.

(1) Let $\X$ be a Serre subcategory of $\mod R$.
According to Lemma \ref{301519}(1), we have $\X=\M^Z(R)$ for some specialization-closed subset $Z$ of $\spec R$.
Putting $\Phi=Z\cap\s(R)$, we easily see that $\Phi$ is a specialization-closed subset of $\spec R$ which is contained in $\s(R)$ and satisfies $\mod^Z(R)=\mod^\Phi(R)$.
Theorem \ref{190906} implies $\res_{\mod R}\X=\mod^\Phi(R)$.

(2)(3) We use the proof of (1).
Set $\Psi=Z\cap\sing R$.
Then $\Psi$ is a specialization-closed subset of $\spec R$ contained in $\sing R$ such that the equalities $\db^Z(R)=\db^\Psi(R)$ and $\ds^Z(R)=\ds^\Psi(R)$ hold.
Hence the surjectivity of the map follows from Theorem \ref{190906}.
\end{proof}

The result below is now ready to be given, which includes Corollary \ref{301201}.
Here, (1) and the equivalence of (b)--(d) in (2) are proved in \cite[Theorem 1.1 and Proposition 4.6]{crs} under the assumption that $R$ is a  Cohen--Macaulay local ring.
Our results show that this assumption is superfluous.

\begin{cor}\label{301355}
\begin{enumerate}[\rm(1)]
\item
The assignments $\Phi\mapsto\mod^\Phi(R)$ and $\X\mapsto\nf(\X)$ make mutually inverse bijections between
\begin{itemize}
\item
the specialization-closed subsets of $\spec R$ contained in $\s(R)$, and
\item
the resolving subcategories of $\mod R$ closed under tensor products and transposes.
\end{itemize}
\item
Let $\X$ be a resolving subcategory of $\mod R$.
Then the following are equivalent:
\begin{enumerate}[\rm(a)]
\item
$\X$ is the resolving closure of a Serre subcategory of $\mod R$;
\item
$\X$ is closed under tensor products and transposes;
\item
$R/\p$ belongs to $\X$ for all $\p\in\nf(\X)$;
\item
For all $\p\in\nf(\X)$ there exists $M\in\X$ such that $\kappa(\p)$ is a direct summand of $M_\p$.
\end{enumerate}
\end{enumerate}
\end{cor}

\begin{proof}
Recall that we have proved in Corollary \ref{191104}(2) that if $R$ is a local ring with maximal ideal $\m$ and residue field $k$, then the equality $\mod^{\{\m\}}(R)=\res_{\mod R}(k)$ holds.
Hence, in view of \cite[Lemma 3.2]{crspd}, we see that all the ten assertions in \cite[Lemma 2.5]{crs} hold without the assumption that $R$ is Cohen--Macaulay.
Therefore, it is observed from \cite[Proposition 3.3]{crspd} and the proof of \cite[Proposition 3.1]{crs} that one can remove from \cite[Proposition 3.1]{crs} the two assumptions that $R$ is local and that $R$ is Cohen--Macaulay.
Thus, the proof of \cite[Theorem 3.3]{crs} actually proves that the statement \cite[Theorem 3.3]{crs} holds without the assumption that $R$ is a Cohen--Macaulay local ring.
Since \cite[Lemma 4.5]{crs} (repsectively, \cite[Lemma 4.4]{crs}) is still valid for an arbitrary commutative noetherian ring (respectively, local ring) $R$, so are \cite[Proposition 4.6 and Theorem 4.7]{crs}.
Now our Theorem \ref{301626} completes the proof of the corollary.
\end{proof}

\begin{ac}
The author is grateful to Osamu Iyama for giving him helpful suggestions and letting him know about \cite{AI}.
The author thanks Takuma Aihara and Xiao-Wu Chen for useful comments.
The author also thanks the referee for reading the paper carefully.
\end{ac}

\end{document}